\documentclass[12pt]{amsart}
\usepackage{amsmath,latexsym,amscd,amsbsy,amssymb,amsfonts,amsthm,fleqn,leqno,
euscript, graphicx, texdraw}

\numberwithin{equation}{section}

\newtheorem{thm}{Theorem}[section]
\newtheorem{pro}[thm]{Proposition}
\newtheorem{lem}[thm]{Lemma}
\newtheorem{cor}[thm]{Corollary}

\theoremstyle{definition}

\theoremstyle{remark}

\begin{document}

\title[Homogeneous $ANR$ spaces and Alexandroff manifolds]
{Homogeneous $ANR$-spaces  and Alexandroff manifolds}

\author{V. Valov}
\address{Department of Computer Science and Mathematics,
Nipissing University, 100 College Drive, P.O. Box 5002, North Bay,
ON, P1B 8L7, Canada} \email{veskov@nipissingu.ca}

\date{\today}
\thanks{The author was partially supported by NSERC
Grant 261914-08.}

 \keywords{absolute neighborhood retracts, cohomological dimension, cohomology groups
homogeneous compacta}

\subjclass[2000]{Primary 55M10, 55M15; Secondary 54F45, 54C55}
\begin{abstract}
We specify a result of Yokoi \cite{yo} by proving that if $G$ is an abelian group and $X$ is a homogeneous  metric $ANR$ compactum with $\dim_GX=n$ and
$\check{H}^n(X;G)\neq 0$, then $X$ is an $(n,G)$-bubble. This implies that any such space $X$ has the following properties:
$\check{H}^{n-1}(A;G)\neq 0$ for every closed separator $A$ of $X$,
and $X$
is an Alexandroff manifold with respect to the class $D^{n-2}_G$ of
all spaces of dimension $\dim_G\leq n-2$. We also prove that
if $X$ is a homogeneous metric continuum with $\check{H}^n(X;G)\neq 0$, then $\check{H}^{n-1}(C;G)\neq 0$ for
any partition $C$ of $X$ such that $\dim_GC\leq n-1$. The last provides a partial answer to a question of Kallipoliti and Papasoglu \cite{kp}.
\end{abstract}
\maketitle\markboth{}{Homogeneous $ANR$}





\section{Introduction}
In this paper we establish some properties of homogeneous metric compacta.
One of the main problems concerning  homogeneous compacta is the Bing-Borsuk \cite{bb} question whether any closet separator
of an $n$-dimensional homogeneous metric $ANR$-space is cyclic in dimension $n-1$. Yokoi's result \cite[Theorem 3.3]{yo} provides a partial answer to this question. Our first result is a clarification of \cite[Theorem 3.3]{yo}, we omit the requirement $G$ to be a principal ideal domain.

\begin{thm}
Let $X$ be a homogeneous  metric $ANR$-continuum with cohomological dimension $\dim_GX=n$ and $\check{H}^n(X;G)\neq 0$, where $G$ is an abelian group. Then $X$ is an $(n,G)$-bubble.
\end{thm}

Following Yokoi \cite{yo}, a compactum $X$ is called an {\em $(n,G)$-bubble} if $\check{H}^n(X;G)\neq 0$ and $\check{H}^n(A;G)=0$ for every closed proper set $A\subset X$. This is a reformulation of the notion of an {\em $n$-bubble} introduced by Kuperberg \cite{kup} and Choi \cite{ch}, see also Karimov-Repov\v{s} \cite{kr} for the stronger notion of an {\em $\check{H}^n$-bubble.}
\begin{cor}
Let $X$ be a homogeneous  metric $ANR$-continuum with $\check{H}^n(X;G)\neq 0$ and $\dim_GX=n$. Then
\begin{itemize}
\item[(i)] $X$ is a strong $V^n_G$-continuum;
\item[(ii)] $X$ is an Alexandroff manifold with respect to the class $D^{n-2}_G$ of
all spaces of dimension $\dim_G\leq n-2$;
\item[(iii)] If $A$ is a closed separator of $X$, then $\check{H}^{n-1}(A;G)\neq 0$.
\end{itemize}
\end{cor}
Item (ii) from Corollary 1.2 was proved in \cite{ktv} under the additional requirement that $G$ is a principal ideal domain.
Here, $\check{H}^n(X;G)$ denotes the reduced $n$-th \v{C}ech cohomology group of $X$ with coefficients from $G$. We say that a set $A\subset X$
is massive if $A$ has a non-empty interior in $X$.

Recall that
a space $X$ is an {\em Alexandroff manifold with respect to a class $\mathcal C$}, see \cite{kktv} and \cite{tv},
 if for every two disjoint, closed massive sets $X_0, X_1\subset X$
there exists an open cover
$\omega$ of $X$ such that there is no partition $P$ in $X$ between
$X_0$ and $X_1$ admitting an $\omega$-map onto a space $Y\in\mathcal C$.
This definition is inspired by the Alexandroff's notion of $V^n$-continua \cite{ps} which is obtained when $\mathcal C$ is
the class of all compacta whose covering dimension is $\leq n-2$.

A compactum $X$ is said to be a {\em $V^n_G$-continuum} \cite{ss}
if for every two open, disjoint subsets $U_1$, $U_2$ of $X$ there exists an open cover $\omega$ of $X_0=X\setminus (U_1\cup U_2)$ such that any
partition $P$ in $X$ between $U_1$ and $U_2$ does not admit an $\omega$-map $g$ of $P$ onto a space $Y$ with
$g^*\colon\check{H}^{n-1}(Y;G)\to\check{H}^{n-1}(P;G)$ being trivial. If, in addition, there exists also an element $\mathrm{e}\in\check{H}^{n-1}(X_0;G)$ such that for any partition $P$ between
$U_1$ and $U_2$ and any $\omega$-map $g$ of $P$ onto a space $Y$ we have $0\neq i^*_P(\mathrm{e})\in g^*\big(\check{H}^{n-1}(Y;G)\big)$, where $i_P$ is the embedding $P\hookrightarrow X_0$, $X$ is called {\em a strong $V^n_G$-continuum}. For example, every $(n,G)$-bubble is a strong $V^n_G$-continuum, see \cite{ktv}.

It follows directly from the above definitions that
$V^n_G$-continua are Alexandroff manifolds with respect to the class  $D^{n-2}_G$ of
all spaces of dimension $\dim_G\leq n-2$. The converse is not true, for example the Menger $n$-dimensional compactum is a $V^n$-continuum, but it is not a $V^n_G$-continuum for any group $G$, see \cite{ktv}.

Homogeneous metric compacta (not necessary $ANR$) are also interesting class of spaces. Krupski \cite{kru} has shown that any
such an $n$-dimensional space is a Cantor $n$-manifold. One of the ingredients of Krupski' proof is
the classical result established by Hurewicz-Menger \cite{hm} and Tumarkin \cite{tu} that any $n$-dimensional compactum contains an
$n$-dimensional Cantor $n$-manifold. Kuz'minov \cite{ku} provided a cohomological counterpart of this fact about $V^n$-continua  (see \cite{kktv} for more general results). Concerning  $V^n_G$-continua we have the following statement which provides a positive answer of Question 4.3 from \cite{ktv}:
\begin{thm}
Any compactum $X$ with $\check{H}^n(X;G)\neq 0$ contains a strong $V^n_G$-continuum.
\end{thm}

Theorem 1.3 could be also compared with the cohomological version of Kuperberg's result \cite[Theorem 5.5]{kup} that any $n$-dimensional $n$-cyclic metric compactum $X$ for which $\check{H}^n(X;\mathbb Z)$ is finitely generated (in particular, any $n$-dimensional $n$-cyclic $ANR$) contains an $n$-bubble.

 The condition $\check{H}^n(X;G)\neq 0$ in the above theorem is essential. For example, let $X$ be the square $\mathbb I^2$ and suppose $X$ contains
 a (strong) $V^2_{\mathbb Z}$-continuum $K$, where $\mathbb Z$ is the group of all integers. Then $\dim K=2$, so $K$ contains a non-empty interior $U$ in $X$. Now, take a segment $\mathbb I$ joining two opposite sides of $\mathbb I^2$ and intersecting $U$. Obviously $\mathbb I\cap K$ is a partition of $K$. Since $\check{H}^{n-1}(P;G)\neq 0$ for every partition $P$ of a $V^n_G$-continuum,
 $\check{H}^1(\mathbb I\cap K;\mathbb Z)\neq 0$. On the other hand, because $\mathbb I\cap K$ is an one-dimensional subset of $\mathbb I$, $\check{H}^1(\mathbb I\cap K;\mathbb Z)=0$.

Kuperberg \cite{kup} asked whether any $n$-dimensional metric compactum contains an $(n-1)$-bubble. This question is still open, but the following corollary provides a result in this direction.
\begin{cor}
Any compactum $X$ with $\dim_GX=n$ contains a strong $V^{n-1}_G$-continuum.
\end{cor}

For finite-dimensional metric compacta and $V^n_G$-continuua Theorem 1.3 and Corollary 1.4 were established in \cite{ss}.

\begin{pro}
Let $X$ be a homogeneous metric continuum with $\check{H}^n(X;G)\neq 0$. Then for
any partition $C$ of $X$ there exists an open cover $\omega$ of $X$ such that $C$ does not admit any  $\omega$-map $g\colon C\to Y$
onto a space of dimension $\dim_GY\leq n-1$ with $g^*:\check{H}^{n-1}(Y;G)\to\check{H}^{n-1}(C;G)$  being a trivial homomorphism.
\end{pro}

Let us note that the $n$-dimensional universal Menger compactum $\mu^n$, which is a homogeneous continuum cyclic in dimension $n$,  contains a separator $C$ such that $\dim C=n$ and $\check{H}^{n-1}(C;G)=0$, see \cite[Corollary 2.6]{ktv}. Therefore, the restriction $\dim_GY\leq n-1$ in Proposition 1.5 and the condition $\dim_GP\leq n-1$ in next corollary are essential.

\begin{cor}
Let $X$ be a homogeneous metric continuum such that $\check{H}^n(X;G)\neq 0$. Then
$\check{H}^{n-1}(P;G)\neq 0$ for every  partition $P$ of $X$ with $\dim_GP\leq n-1$.
\end{cor}

Kallipoliti and Papasoglu \cite{kp} have shown that every 2-dimensional locally connected, simply connected homogeneous metric continuum can not be separated by an arc, and asked if the simple connectedness can be dropped from this result.
Corollary 1.6 provides a partial answer to the Kallipoliti-Papasoglu question.

\section{Cohomological carriers}

In this section we consider cohomological carriers of non-trivial elements of $\check{H}^n(X;G)$ and establish some properties of them. We fix an abelian group $G$, an integer $n$ and a metric compactum $X$ with $\check{H}^n(X;G)\neq 0$. A closed non-empty set $A\subset X$ is said to a {\em a cohomological carrier}  (shortly, a carrier) of a non-zero element $\alpha\in\check{H}^n(X;G)$ if $i^*_A(\alpha)\neq 0$ and $i^*_B(\alpha)=0$ for every proper closed subset $B\subset A$, where $i_A$ denotes the inclusion map $A\hookrightarrow X$.

\begin{lem}
For every non-zero element $\alpha\in\check{H}^n(X;G)$ there exists a carrier. Moreover, if $A$ is a carrier of $\alpha$, then  $\check{H}^{n-1}(P;G)\neq 0$ for any closed partition $P$ of $A$.
\end{lem}

\begin{proof}
The first part of Lemma 2.1 follows from Zorn's lemma and the continuity of \v{C}ech cohomology. For the second part, suppose $A$ is a carrier of $\alpha$ and $P$ a partition of $A$. Then there exist two closed proper subsets $A_1$ and $A_2$ of $A$ such that $A=A_1\cup A_2$ and $P=A_1\cap A_2$. Consider the
the Mayer-Vietoris exact sequence
{ $$
\begin{CD}
\check{H}^{n-1}(P;G)\rightarrow\check{H}^{n}(A;G)\rightarrow\check{H}^{n}(A_1;G)\oplus\check{H}^{n}(A_2;G).
\end{CD}
$$}\\
For every $i=1,2$ let $\partial_i\colon\check{H}^{n}(A;G)\to\check{H}^{n}(A_i;G)$ be generated by the inclusion
$A_i\hookrightarrow A$. Denote also by $\triangle$ and $\varphi$, respectively, the left and right homomorphism of the above sequence.
Since each $A_i$ is a proper subset of $A$ we have $\varphi(\beta)=(\partial_1(\beta),\partial_2(\beta))=0$, where
$\beta=i^*_A(\alpha)$. So, there exists $\gamma\in\check{H}^{n-1}(P;G)$ with $\triangle(\gamma)=\beta$. Because $\beta$ is a non-trivial element of $\check{H}^{n}(A;G)$, so is $\gamma$. Hence, $\check{H}^{n-1}(P;G)\neq 0$.
\end{proof}

Everywhere below, if $B\subset A$, then $i_{A,B}$ denotes the inclusion $B\hookrightarrow A$. The next lemma is an analogue of Lemma 4 from \cite{ch}.
\begin{lem}
Let $A\subset X$ be a carrier of a non-trivial element $\alpha\in\check{H}^n(X;G)$ and $B$ a closed subset of $X$. Then $A\subset B$ if and only if $\mathrm{Ker}(j_B^*)\subset\mathrm{Ker}(j_A^*)$, where
$j_A=i_{A\cup B,A}$ and $j_B=i_{A\cup B,B}$ are the corresponding inclusions.
\end{lem}

\begin{proof}
Obviously $A\subset B$ implies $\mathrm{Ker}(j_B^*)\subset\mathrm{Ker}(j_A^*)$. Suppose that $\mathrm{Ker}(j_B^*)\subset\mathrm{Ker}(j_A^*)$, but $B$ does not contain $A$. Then $A\cap B$ is a proper closed subset of $A$ (possibly empty). The left homomorphism in the Mayer-Vietoris exact sequence
{ $$
\begin{CD}
\check{H}^{n}(A\cup B;G)\rightarrow\check{H}^{n}(A;G)\oplus\check{H}^{n}(B;G)\rightarrow\check{H}^{n}(A\cap B;G).
\end{CD}
$$}\\
 is defined by $(j_A^*,j_B^*)$, while the right one $i^*$
assigns to each $(\beta_1,\beta_2)\in\check{H}^{n}(A;G)\oplus\check{H}^{n}(B;G)$ the difference $i_{A,A\cap B}^*(\beta_1)-i_{B,A\cap B}^*(\beta_2)$. Since $A\cap B$ is a proper subset of $A$, $i_{A\cap B}^*(\alpha)=0$. Then $i^*((i^*_A(\alpha),0))=0$. Consequently, there exists
 $\beta\in\check{H}^{n}(A\cup B;G)$ with $(j_A^*(\beta),j_B^*(\beta))=(i^*_A(\alpha),0)$. So, $\beta\in\mathrm{Ker}(j_B^*)$ and, according to our assumption, $\beta\in\mathrm{Ker}(j_A^*)$. The last relation contradicts $i^*_A(\alpha)\neq 0$. Therefore, $A\subset B$.
\end{proof}

The next proposition is actually Theorem 5 from \cite{ch}. We provide a different proof of that theorem.

\begin{pro}
Let $A\subset X$ be a carrier for a non-trivial element of $\check{H}^{n}(X;G)$ and $f\colon X\to X$ a map homotopic to the identity on $X$. If $\dim_GX\leq n$, then $A\subset fA$.
\end{pro}

\begin{proof}
By \cite{hu}, we can identify the cohomological group $\check{H}^{n}(A\cup fA;G)$ with the group  $[A\cup fA, K(G,n)]$ of homotopy classes from $A\cup fA$ to $K(G,n)$, where $n>0$ and $K(G,n)$ denotes an Eilenberg-MacLane complex. Similarly,
$\check{H}^{n}(A;G)$ and $\check{H}^{n}(fA;G)$ are identified
with the groups $[A, K(G,n)]$ and $[fA, K(G,n)]$, respectively.

By Lemma 2.2, it suffices to prove that if $\alpha\in\mathrm{Ker}(j_{fA}^*)$ then $\alpha\in\mathrm{Ker}(j_{A}^*)$. So, we fix $\alpha\in\check{H}^{n}(A\cup fA;G)$ with $j_{fA}^*(\alpha)=0$. According to the above identifications, there exists a map $g\colon A\cup fA\to K(G,n)$ such that $\alpha=[g]$. Since $\dim_GX\leq n$, $g$ can be extended to a map $\tilde{g}\colon X\to K(G,n)$.
Because $j_{fA}^*(\alpha)=0$, we can find a homotopy $H_1\colon fA\times [0,1]\to K(G,n)$ with $H_1(x,0)=g(x)$ and $H_1(x,1)=*$ for all $x\in fA$, where $*$ is a point from $K(G,n)$. Then
the homotopy $H_2\colon A\times [0,1]\to K(G,n)$, $H_2(x,t)=H_1(f(x),t)$, connects the constant map $\kappa\colon A\hookrightarrow *$ and the map $h\colon A\to K(G,n)$ defined by $h(x)=g(f(x))$. Next, consider a homotopy $F\colon A\times [0,1]\to X$ with $F(x,0)=f(x)$ and $F(x,1)=x$, and define  $H\colon A\times [0,1]\to K(G,n)$ by $H(x,t)=\tilde{g}(F(x,t))$. We have $H(x,0)=g(f(x))$ and $H(x,1)=g(x)$ for all $x\in A$. Hence, $H$ is connecting the maps $h$ and the restriction $g|A$ of $g$ over $A$. Finally, combining $H$ and $H_2$, we can produce a homotopy on $A$ connecting the map $g|A$ and the constant map $\kappa$. Hence, $j_{A}^*(\alpha)=[g|A]=0$.
\end{proof}

Before proving the next property of carriers, we introduce some more notations. If $\omega$ is a finite open cover of  a closed
set $Z\subset X$, we denote by $|\omega|$ and $p_\omega$, respectively, the nerve of $\omega$ and a map from $Z$ onto $|\omega|$ generated by a
partition of unity subordinated to $\omega$.
Furthermore, if $C\subset Z$ and $\omega(C)=\{W\cap C: W\in\omega\}$, then $p_{\omega(C)}\colon C\to |\omega(C)|$ is the restriction
$p_\omega|C$.
Recall also that $p_\omega$ generates maps $p_\omega^*\colon\check{H}^k(|\omega|;G)\to\check{H}^k(Z;G)$ for $k\geq 0$. Moreover, if
$q_\omega\colon Z\to|\omega|$ is a map generating by (another) partition of unity subordinated to $\omega$, then $p_\omega$ and $q_\omega$
are homotopic. So, $p_\omega^*=q_\omega^*$.

\begin{pro} Let $K$ be a carrier for a non-trivial element of $\alpha\in\check{H}^{n}(X;G)$. Then for any two open disjoint subsets $U_1$ and $U_2$ of $K$ there exists an open cover $\omega$ of $K\setminus (U_1\cap U_2)$ and an element
$\eta\in\check{H}^{n-1}(|\omega|;G)$ such that $p_{\omega(C)}^*(i_{\omega(C)}^*(\eta))\neq 0$
for every partition $C$ of $K$ between $U_1$ and $U_2$, where $i_{\omega(C)}$ is the inclusion $|\omega(C)|\hookrightarrow|\omega|$.
\end{pro}

\begin{proof}
Let $K_1=K\setminus U_1$, $K_2=K\setminus U_2$, $C$ be a partition of $K$ between $U_1$ and $U_2$, and $F_1$, $F_2$ closed
subsets of $K$ such that: $F_1\cap F_2=C$, $F_1\cup F_2=K$, $F_1\subset K_1$ and $F_2\subset K_2$. Consider the commutative diagram whose rows
are Mayer-Vietoris sequences:
{ $$
\begin{CD}
\check{H}^{n-1}(K_1\cap K_2;G)@>{{\delta}}>>\check{H}^{n}(K;G)@>{{j}}>>\check{H}^{n}(K_1;G)\oplus\check{H}^{n}(K_2;G)\\
@ VV{i^*}V
@VV{id}V@VV{i_1^*\oplus i_2^*}V\\
\check{H}^{n-1}(C;G)@>{{\delta_1}}>>\check{H}^{n}(K;G)@>{{j_1}}>>\check{H}^{n}(F_1;G)\oplus\check{H}^{n}(F_2;G).
\end{CD}
$$}\\

Since $j(\beta)=0$, where $\beta=i^*_K(\alpha)$, there exists a non-zero element $\gamma\in\check{H}^{n-1}(K_1\cap K_2;G)$ with $\delta(\gamma)=\beta$. Consequently, we can find an open cover $\omega$ of $K_1\cap K_2$ and a non-trivial element $\eta\in\check{H}^{n-1}(|\omega|;G)$ such that $p_\omega^*(\eta)=\gamma$. It follows from the commutativity of the above diagram that $i^*(\gamma)\neq 0$. Then the equality $p_{\omega(C)}^*(i_{\omega(C)}^*(\eta))=i^*(\gamma)$ completes the proof.
\end{proof}

\begin{pro}
Every carrier for a non-trivial element of $\check{H}^{n}(X;G)$ is a strong $V^n_G$-continuum.
\end{pro}

\begin{proof}
Indeed, suppose $U_1$ and $U_2$ are open subsets of $K$ having disjoint closures. Let $\omega$ be an open cover of $K_0=K\setminus (U_1\cup U_2)$ and $\mathrm{e}\in\check{H}^{n-1}(|\omega|;G)$ a non-trivial element satisfying the hypotheses of Proposition 2.4. Assume $C$ a partition  of $K$ between $U_1$ and $U_2$ admitting an $\omega$-map $g$ onto a space $T$. Thus, we can find a finite open cover $\tau$
of $T$  such that $\nu=g^{-1}(\tau)$ is refining $\omega$. Let $p_\nu\colon C\to|\nu|$ be a map onto the nerve of $\nu$ generated by a partition
of unity subordinated to $\nu$. Obviously, the function $V\in\tau\rightarrow g^{-1}(V)\in\nu$ provides a
a simplicial homeomorphism $g^{\tau}_\nu\colon|\tau|\to|\nu|$. Then the maps $p_\nu$ and
$g_\tau=g^{\tau}_\nu\circ\pi_\tau\circ g$, where $\pi_\tau\colon T\to|\tau|$ is a map generated by a partition of unity
subordinated to $|\tau|$,
are homotopic. Hence, $p_{\nu}^*=g^*\circ\pi_{\tau}^*\circ (g^{\tau}_\nu)^*$.

On the other hand, since $\nu$
refines $\omega$, we can find a map $\varphi_\nu\colon|\nu|\to |\omega(C)|$ such that $p _{\omega(C)}$ and $\varphi_\nu\circ p_{\nu}$
are homotopic. Therefore, $p_{\omega(C)}^*=p_{\nu}^*\circ\varphi_\nu^*$.
According to Proposition 2.4, there exists $\eta\in\check H^{n-1}(|\omega|;G)$ with $p_{\omega(C)}^*(i_{\omega(C)}^*(\eta))\neq 0$. Since
$i_C^*(p_\omega^*(\eta))=p_{\omega(C)}^*(i_{\omega(C)}^*(\eta))$, $\mathrm{e}=p_\omega^*(\eta)$ is a non-zero element of
$\check H^{n-1}(K_0;G)$. Here $p_\omega\colon K_0\to |\omega|$ is a map generated by a
partition of unity subordinated to $\omega$ and $i_C:C\hookrightarrow K_0$ is the inclusion map. Moreover, both equalities $p_{\nu}^*=g^*\circ\pi_{\tau}^*\circ (g^{\tau}_\nu)^*$ and $p_{\omega(C)}^*=p_{\nu}^*\circ\varphi_\nu^*$ yield that $i_C^*(\mathrm{e})$ is a non-trivial element of $g^*(\check H^{n-1}(T;G))$.
\end{proof}

\section{Proof of Theorem 1.1 and Corollary 1.2}

\textit{Proof of Theorem $1.1$.} Suppose $G$ is an abelian group, $X$ is a non-trivial homogeneous  metric $ANR$-continuum with $\dim_GX=n$ and $\check{H}^n(X;G)\neq 0$. Since $X$ is an $ANR$, $n\geq 1$ and there exists a positive  $\epsilon$ such that any two $\epsilon$-close maps from $X$ into $X$ are homotopic (we say that two maps $f_1, f_2\colon X\to X$ are $\epsilon$-close if $dist(f_1(x),f_2(x))<\epsilon$ for each $x\in X$).

It suffices to show that if $A$ is a carrier for a non-trivial element $\alpha\in\check{H}^n(X;G)$, then $A=X$. Indeed, suppose there exists a proper subset $B\subset X$ with $\check{H}^n(B;G)\neq 0$, and choose a non-trivial
element $\beta\in\check{H}^n(B;G)$.  Since $\dim_GX=n$, there exists $\alpha\in\check{H}^n(X;G)$ with $\beta=i_{B}(\alpha)$. Because the carrier of $\alpha$ is $X$ and $B$ is a proper subset of $X$, $i_{B}(\alpha)=0$, a contradiction.

Next, suppose $A\subset X$ is a carrier for a non-trivial $\alpha\in\check{H}^n(X;G)$ and $A$ is a proper set. According to the Effros'
theorem \cite{ef}, there corresponds a positive number $\delta$ with the following property: whenever $x$ and $y$ are points from $X$ and
$dist(x,y)<\delta$, there is a homeomorphism $h\colon X\to X$ such that $h(x)=y$ and $h$ is $\epsilon$-close to the identity $id_X$ on $X$. Because $A\neq X$, we can choose points $a\in A$ and $b\not\in A$ with $dist(a,b)<\delta$. Consequently, there would be a homeomorphism $f\colon X\to X$ such that $f(a)=b$ and $dist(f,id_X)<\epsilon$. Obviously, $g=f^{-1}\colon (X,fA)\to (X,A)$ generates the isomorphism
$g^*\colon\check{H}^n(X;G)\to\check{H}^n(X;G)$ and
$B=f(A)$ is a  carrier for the element $g^*(\alpha)$. Moreover, the homeomorphism $g$ is also $\epsilon$-close to $id_X$. Hence, $g$ is homotopic to $id_X$. Applying Proposition 2.3 to the carrier $B$ and the homeomorphism $g$, we obtain that $B\subset g(B)=A$ which contradicts $b\in B\setminus A$. Hence, $A$ should be the whole space $X$. \hfill$\square$

\textit{Proof of Corollary $1.2$.} Item (i) follows from Proposition 2.5. Item (ii) follows from the simple observation that
any $V^n_G$-continuum is an Alexandroff manifold with respect to the class $D^{n-2}_G$. Since $X$ is a carrier for every non-trivial
$\alpha\in\check{H}^n(X;G)$, item (iii) follows from Lemma 2.1. \hfill$\square$

\section{$V^n_G$-continua}

In this section we provide the proofs of Theorem 1.3, Corollary 1.4 and Propositions 1.5-1.7.

\textit{Proof of Theorem $1.3$.}
Since $\check H^{n}(X;G)\neq 0$, $X$ contains a carrier $A$ of a non-trivial element of $\check H^{n}(X;G)\neq 0$. Then, according to Propositions 2.5, $A$ is a strong $V^n_G$-continuum. \hfill$\square$

\textit{Proof of Corollary $1.4$.} This corollary follows directly from Theorem 1.3 because every compactum $X$ with $\dim_GX=n$ contains
a closed subset $F$ such that $\check{H}^{n-1}(F;G)\neq 0$ (see, for example, \cite{ku}). \hfill$\square$

\textit{Proof of Proposition $1.5$.} Suppose there exists a partition $C$ of $X$ such that for every open cover $\omega$ of $X$, $C$ admits an $\omega$-map $g_\omega\colon C\to Y_\omega$
onto a space of dimension $\dim_GY_\omega\leq n-1$ with $g_\omega^*:\check{H}^{n-1}(Y_\omega;G)\to\check{H}^{n-1}(C;G)$  being a trivial homomorphism. Then, according to \cite[Theorem 2.4]{cv}, $\dim_GC\leq n-1$. Obviously, the boundary $B$ of $C$ in $X$ is also a partition of $X$ and $\dim_GB\leq n-1$. Moreover, we have the commutative diagram below, where $g_\omega|B\colon B\to g_\omega(B)$ is the restriction of $g_\omega$
{ $$
\begin{CD}
\check{H}^{n-1}(Y_\omega;G)@>{{g^*_\omega}}>>\check{H}^{n-1}(C;G)\\
@VV{i^*_{g(B)}}V@VV{i_B^*}V\\
\check{H}^{n-1}(g_\omega(B);G)@>{{(g_\omega|B)^*}}>>\check{H}^{n-1}(B;G)@.
\end{CD}
$$}\\
Since $\dim_GY_\omega\leq n-1$, $i^*_{g(B)}$ is a surjection.  This implies that $(g_\omega|B)^*$ is the trivial homomorphism because so is $g_\omega^*$. Therefore, considering $B$ instead of $C$, we may assume that $C$ does not have interior points in $X$.
The above diagram also shows that for every closed subset $A\subset C$ and every $\omega$ the restriction $g_\omega|A$ is an $\omega$-map onto
$g_\omega(A)$ such that $(g_\omega|A)^*:\check{H}^{n-1}(g_\omega(A);G)\to\check{H}^{n-1}(A;G)$ is the trivial homomorphism.

By Theorem 1.3, there exists a strong $V^n_G$-continuum $K\subset X$.
Since $X$ is homogeneous, we may also assume that $K\cap C\neq\varnothing$. Observe that $z\in K\setminus C$ for some $z$.
Indeed, the inclusion $K\subset C$ would imply that if $P$ is a partition of $K$ and $\gamma$ any open cover of $K$, then $P$ admits an $\gamma$-map $h_\gamma$ onto a space $T$  such that $(h_\gamma)^*\colon\check{H}^{n-1}(T_\gamma;G)\to\check{H}^{n-1}(P;G)$ is trivial. This would contradict the fact that $K$ is a strong $V^n_G$-continuum. Let $X\setminus C=U\cup V$ and $z\in V$,
where $U$ and $V$ are nonempty, open and disjoint sets in $X$. Then the Effros theorem \cite{ef} allows us to push $K$ towards $U$ by a
small homeomorphism $h\colon X\to X$ so that the image $h(K)$ meets both $U$ and $V$ (see the proof of Lemma 2 from \cite{kru}
for a similar application of Effros' theorem). Therefore, $S=h(K)\cap C$ is a partition of $h(K)$ such that for any $\omega$ the restriction $g_\omega|S$ is an $\omega$-map generating a trivial homomorphism $(g_\omega|S)^*$, a contradiction. \hfill$\square$

\textit{Proof of Corollary $1.6$.} It follows directly from Proposition 1.5. \hfill$\square$

\smallskip
\textbf{Acknowledgments.} The author would like to express his
gratitude to K. Kawamura and K. Yokoi for providing some information. The author also thanks the referee for his/her valuable remarks and suggestions which improved the paper.


\end{document}